\documentclass[12pt]{article}
\usepackage{amsfonts,amsthm,amsmath,amssymb}
\newtheorem{Theorem}{Theorem}[section]
\newtheorem{Lemma}[Theorem]{Lemma}
\theoremstyle{definition}

\newtheorem{Remark}[Theorem]{Remark}

\newcommand{\sEnd}{\mbox{${\sE}{nd}$}}

\newcommand{\boxtensor}{{\Box\kern-9.03pt\raise1.42pt\hbox{$\times$}}}

\newcommand{\tensor}{\otimes}

\newcommand{\sE}{{\mathcal E}}

\newcommand{\sO}{{\mathcal O}}

\numberwithin{equation}{section}
\newcounter{elno}                

\newcounter{example}[section]

\title{Some Further Remarks on the Local Fundamental Group Scheme}

\author{V.B. Mehta\protect\footnote{email:vikram@math.tifr.res.in}}

\begin{document}
\maketitle

\begin{abstract}
  We prove that the Local  Fundamental Group Scheme satisfies the Lefschetz-
Bott theorems in characteristic $p$. The proofs are standard applications 
of the Enriques-Severi-Zariski-Serre vanishing theorems and known facts 
about the $p$-curvature .
  
\end{abstract}



\date{}

\section{Introduction}
 
The results of this paper were already proved by Indranil Biswas and 
Yogesh Holla ( arXiv.math /0603299 v1 [math AG], 13 March 2006,  
arXiv.math /0603299v1 [math AG], 1st May 2007, and finally published in 
the Journal of Algebraic Geometry, Vol. 16, No.3, 2007, pages [547-597].
The present work was done later ( around late 2006-early 2007) but 
independently. A preliminary version of the present work was put on 
the arXiv in January 2007. This has been withdrawn as soon as we were 
informed that Biswas  and Holla had already put up their paper on the 
arXiv in March 2006. 

However since our proofs are shorter and more suited for future applications 
( "On the Grothendieck-Lefschetz Theorem for a family of Varieties " ,
with Marco Antei, in preparation ),  we are putting a shortened version 
on the arXiv.

The Fundamental Group Scheme was introduced by Madhav Nori in [N1,N2].
 In order to prove the conjectures made in [loc.cit], S.Subramanian
 and the present author had introduced the "Local Fundamental Group 
 Scheme ", denoted by $\pi^{loc}(X)$[MS1], which is the infinitesimal 
part of $\pi(X)$,[MS2]. We had also introduced the notion of a 
\emph{$F$-trivial} vector bundle on a variety in characteristic 
$p$[loc.cit.]  One  may ask whether the Lefschetz-Bott theorems  hold for  
$\pi^{loc}(X)$. In  other words, let $X$ be a smooth projective  variety 
over an algebraically  closed filed of characteristic $p$,and let
 $H$ be a very ample line bundle on $X$. The question is if $\pi^{loc}(Y)
 \to \pi^{loc}(X)$ is a {\it surjection} if dim $X \geq 2$ and degree $Y
 \geq n_0$, where $n_0$ is an integer depending {\it only} on $X$.
 Similiarly , if dim $X \geq 3$,the question is if $\pi^{loc}(Y) \to
 \pi^{loc}(X)$ is an {\it isomorphism} if deg $Y \geq n_1$, where $n_1$
 is an integer depending {\it only} on $X$.

  We give a positive answer to   both these  questions. The methods  
only involve applying the lemma of Enriques-Severi-Zariski-Serre ( E-S-Z-S) 
over and over again. We also heavily use the well-known facts about the 
$p$-curvature  of 
integrable connections in characteristic $p$ [K1].  The 
notation is the same as in [MS2],which is briefly
 recalled. We thank Vijaylaxmi Trivedi and Vittorio for helpful 
discussions. This work was completed while the author was visiting the ICTP, 
Trieste as a Senior Research Associate. We would like to thank the ICTP 
for its support and hospitality during that period. Our method of proof
is heavily influenced by a paper of Karen Smith  [S], especially
[Thm. 3.5].  All the facts about Tannaka Categories that we 
use may be found in [N1]. 

\section{ The first theorem}

\begin{Theorem}\label{t1}Let $X$ be a non-singular projective variety over
 an algebraically closed field of characteristic  $p$,of dimension $\geq 
2$. Let $H$ be a very ample line bundle on $X$. Then there exists an 
integer $n_0$, depending {\it only} on $X$, such that for any smooth $Y 
\in |nH|$, of deg $n \geq n_0$, the canonical map : $\pi^{loc}(Y)  \to 
 \pi^{loc}(X)$ is a surjection.
\end{Theorem}
 
Before we begin the proof, we recall some facts from [MS2]. Let $F: X
 \to X$ be the Frobenius map. For any integer $t \geq 1$, denote by $C_t$
the category  of all $V \in Vect(X)$ such that $F^{t*}(V)$ is the  
 { \it trivial} vector bundle on $X$. Let $FT(X)$ denote the union of 
  $C_t(X)$ for all $t \geq1$. Fix a  base point $x_0 \in X$,
Consider the functor $S : C_t \to Vect (X)$ given by $V \to V_{x_0}$  
It is seen that $C_t$ , with the fibre functor $S$, is a Tannaka
category [MS1,MS2]. denote the corresponding Tannaka group by $G_t$. 
It is also easily seen that $$\pi^{loc}(X) \simeq \lim_{\leftarrow t} 
G_t$$, where $\pi^{loc}(X)$ is the Tannaka group associated to the
category $(V \in Vect (X)\ F^t(V)$ is trivial for some $t$.

Now let $t = 1$, and consider $G_1(Y)$ and $G_1(X)$. Let $V \in C_1(X)$.
with $V$ stable.$V$ corresponds to a principal H bundle $E \rightarrow X$ 
,where $E$ is {\it reduced} in the sense of Nori [N1,page 87,Prop.3], or 
just N-reduced. 
To prove that $G_1(Y) \to G_1(X)$ is a surjection , it is enough to 
prove that $V/Y$ is {\it stable}, or better still that $E/Y$ is 
{\it N-reduced}.

\begin{Lemma}\label{2}
 If X and E are as above , then there exists an integer $n_0$,
 depending only on $X$ ,such that for all $n \geq n_0$, and for all 
 smooth $Y \in |nH|$, $E/Y$ is N-reduced.
\end{Lemma}

\begin{proof}

Let $f: E \rightarrow X$ be given.Then $f_*{\sO_E}$ belongs to $C_1$.
Denote it by $W$ for simplicity. It is easy to see that $W \in FT(X)$.
Also note that for any $V \in Vect (X), V \in FT(X)$ if and only if 
 $V^* \in  FT(X)$. 
Now  consider $$0 \rightarrow {\sO_X}
 \rightarrow F_*{\sO_X} \rightarrow B^1 \rightarrow 0\eqno{(1)}$$,
 Tensor (1) with $W^*(-n)$. There is an $n_0$ such that for 
 $n \geq n_0$, we have $Hom(W(n), B^1) = 0$. In fact, we have that 
$Hom(V(n),B^1)= 0$ for all $V \in FT(X)$, for all $ n > n_0$, where $n_0$
is \emph {independent} of $V$ in $FT(X)$. So the canonical map
 $H^1(W(-n) \rightarrow H^1(F^*(W(-n))$ is injective. But the last space 
  is just $H^1({\sO_X}(-np))^r$, $r = rank W$.  But this is $0$ as soon 
as $n \geq n_1$, independent of $V \in C_1(X)$, as dim $X \geq 2$. Hence 
$H^0(Y,W/Y) =1$,  which proves that  $E$ restricted to $Y$ is also 
N-reduced.

Now assume $t \geq 2$. We shall assume in fact that $t = 2$, because an 
 identical proof works for $t \geq 3$.Consider the map $F^2 : X 
\rightarrow X$. Let $B_2^1$ be the cokernel.We have the exact sequence$$
0 \rightarrow B^1 \rightarrow B_2^1  \rightarrow F_*B_1 \rightarrow 
0\eqno{(2)}$$
Let $W \in C_2(X)$ and tensor (2) with $W(-n)$. One sees 
immediately that if $H^1(W(-n){\tensor}B^1) = 0$  for all $n \geq n_0$,
then also $H^1((W(-n){\tensor}B_2^1) = 0$  for all $n \geq n_0$, for the
{\it same} integer $n_0$. So if  E is N-reduced on $X$, then E remains
 N-reduced on $Y$, if deg $Y \geq n_0$ , for the same integer $n_0$, which
 worked for $C_1(X)$. The proof for bigger $t$ goes the same way , by 
 taking the cokernel of $F^t$, where $t \geq3$. This concludes the proof 
of Theorem 2.1. 
\end{proof}

\section{The second theorem}

\begin{Theorem}\label{t2}Let $X$ be smooth and projective of dim $\geq 3$. 
Then there exists an integer $n_0$, depending only on $X$ ,such that for 
any smooth $Y \in |nH|,n \geq n_0$, the canonical map $\pi^{loc}(Y)
\to \pi^{loc}(X)$ is an isomorphism. 
 \end{Theorem}

\begin{Remark}In the sequel we shall use the phrase " for a uniform $n$" 
 or" there exists a uniform integer $n$" to denote a positive integer $n$ 
,which may depend on $X$ , but not on $V$ , for $V \in FT(X)$.
\end{Remark}

Before we begin the proof we need the following lemma:

\begin{Lemma}With X as above, then there exists a uniform integer $n_0$
 such that  we have 
$H^1(\Omega^1_X( -n) \otimes V) = 0$, for all $V \in FT(X)$, for all $n  
\geq n_0$. 
\end{Lemma}

\begin{proof}

Let $T_X$ be the tangent bundle of $X$. For some $s$ , depending only on
$X$, $T_X(s)$ is generated by global sections. So we have $$0 \rightarrow
S^* \rightarrow \sO_X^N  \rightarrow T_X(s) \rightarrow 0.$$  Dualizing,we 
get $$ 0 \rightarrow \Omega^1(-s) \rightarrow \sO_X^N \rightarrow S 
\rightarrow   0 \eqno{(1)}$$.Tensor $(1)$ with $V(-t)$,to get $$ 0 
\rightarrow V(-t) \otimes \Omega^1(-s) \rightarrow  V(-t)\otimes \sO_X^N 
\rightarrow  V(-t)\otimes S \rightarrow 0\eqno{(2)}$$. Applying $F^r$ to 
$(2)$
we get $$ 0 \rightarrow [V(-t) \otimes \Omega^1(-s)]^{p^r} \rightarrow 
[V(-t)\otimes \sO_X^N]^{p^r} \rightarrow [V(-t) \otimes S]^{p^r} 
\rightarrow 
0\eqno{(3)}$$. It is clear that $H^0(V(-t) \otimes S) = 0$ for $t >> 0$,
 \emph{independent} of $V$. Now look at $$H^1(V(-t) \otimes \Omega^1(-s))
\rightarrow H^1(V(-t) \otimes \sO_X^N) \rightarrow H^1(V(-t) 
\otimes \sO_X^N)^{p^r} \rightarrow 0\eqno{(4)}$$ where the right arrow is 
the 
map induced by $F^r$.But one knows that $H^1(V(-t) \otimes \sO_X^N)^{p^r}$
vanishes for $t >> 0$, for a uniform $t$.  
And $H^1[V(-t)) \otimes \Omega^1(-s)] \rightarrow  H^1[V(-t) \otimes 
\sO_X^N]$ is injective for $t >> 0$, for a uniform $t$ . 

 Also $H^1[V(-t) \otimes \sO_X^N] \rightarrow  H^1[V(-t) \otimes
\sO_X^N]^{p^r}$  is injective for $t >> 0$, for a uniform $t$. 
 Hence $H^1(V(-t) \otimes \Omega^1)$ vanishes for all $t >> 0$,
 for a uniform $t$. This concludes the proof of Lemma 3.3. 
\end{proof}

\begin{Remark} In fact ,the above proves the following : Let W be an 
arbitrary vector bundle on $X$. Then there exists a uniform $t_0$, such 
that for all $t \geq t_0$, and for all $V \in FT(X)$,we have $H^1( V 
\otimes W(-t)) = 0$. This lemma  and remark are the \emph {key} points in 
the  paper, and will be used over  and over again, without explicit 
mention.   
\end{Remark}

\begin{Lemma}\label{3}Any $W \in C_t(Y)$ lifts uniquely to an element
$V \in C_t(X)$ ,if $degree Y = n >> 0$, n is uniform.
\end{Lemma}

\begin{proof}
  First ,we give an idea of the proof We assume that dimension $X \geq 3$, 
and we pick an arbitrary smooth $Y 
\in |nH|$. First assume $t = 1$.Take a  $W \in C_1(Y)$ and we show that it 
lifts uniquely to $V \in C_1(X)$ . Such a $W$ is trivialized by the 
Frobenius,so there  exists $M$, an $r \times r$ matrix of 1- forms on 
$Y$,which gives  an integrable connection $\nabla$ on $\sO^r_Y$, with $p$- 
curvature $0$. Consider $$ 0 \rightarrow \frac{I}{I^2} \rightarrow  
\Omega^1_X/Y \rightarrow  \Omega^1_Y \rightarrow  0\eqno{(1)}$$ and $$0 
\rightarrow \Omega_X^1(-n) \rightarrow \Omega^1_X  \rightarrow 
\Omega^1_X/Y \rightarrow 0\eqno{(2)}$$
 
Here $I$ is the ideal sheaf of $Y, I = \sO_X(-n)$.  We get an integer 
$n_0$,  such that for all $n \geq n_0, H^0(X,\Omega^1_X) \to 
H^0(Y,\Omega^1_Y)$ is an isomorphism. So $M$ lifts to $\ M_1$, a  $r 
\times r$ matrix  of $1$ forms on $X$.Now  $\nabla$ has $p$-curvature $0$ 
on $Y$.We now show that $\nabla_1$ defined by $M_1$ has $p$-curvature $0$ 
on $X$. The  curvature of $\nabla$ is an element of $H^0(End \sO^r_Y 
\otimes \Omega^2_Y)$, Again by E-S-Z-S, the curvature of $\nabla_1$ is $0$ 
if degree $Y \geq n_1$, for some integer  $n_1$, depending only on $X$. 
The  $p$-curvature of $\nabla_1$ is an element of $H^0(F^*\Omega^1_X
\otimes End \sO^r_X)$, which again vanishes if deg $Y \geq n_2$. So any
element $W \in C_1(Y)$ lifts uniquely to an element $V \in C_1(X)$.

For $t > 1$,we use induction on $t$. Assume that for lower values of $t$,
any element $W \in C_t(Y)$ has been lifted uniquely to an element $V \in
C_t(X)$ ,where degree $Y = n$, where $n$ is uniform.We show in fact 
that there exists a uniform $n$ with the property:if $W \in C_{t +1}(Y)$, 
then $W$ lifts to $X$. This proceeds by showing that $F^*W$
lifts uniquely to $X$, and this lifted bundle, say $V_t$,on $X$ has an 
integrable $p$-flat connection. So on $X$, $V_t$ will descend 
under $F$ to $V_{t+1}$. And $V_{t+1}$ restricts to $W_{t+1}$ on $Y$. We 
begin the proof by  establishing a couple of claims: 

Claim (1) :  There exists a uniform $n$ such that for $V \in FT(X)$ and  
 for a $Y \in |nH|$   , we have $H^0(\sEnd V \otimes 
\Omega^1_Y(- n)) = 0$.
 
Proof of Claim 1): Look at 
$$0 \rightarrow  \sO_Y(-n) \rightarrow  \Omega^1_X/Y \rightarrow 
\Omega^1_Y \rightarrow  0\eqno{(1)}$$        
and 
$$ 0 \rightarrow \Omega^1_X(-n) \rightarrow  \Omega^1_X  \rightarrow
\Omega^1_X/Y \rightarrow 0\eqno{(2)}$$

 Tensor (2) with $\sEnd V(-n)$,  we get $H^0(\sEnd V(-n) \otimes 
\Omega^1_X/Y)
= 0$ ,for a uniform $n$. From (1) after tensoring  with $\sEnd V$, one 
sees  that $H^1(\sEnd V \otimes \sO_Y( - 2n)) = 0$ implies that 
$H^0(\sEnd V( - n) \otimes \Omega^1_Y)) = 0$. All this is for a uniform 
$n$.  Now we prove that $H^1(\sEnd V \otimes \sO_Y(- 2n)) = 0$ , for a
uniform $n$, in

Claim (2); One has $H^1(\sEnd V \otimes \sO_Y( - 2n) = 0$,
for any $V \in FT(X)$ and smooth $Y \in |nH|$, for a uniform $n$.

Proof of Claim (2): Look at $$0 \rightarrow  \sO_X \rightarrow F_*\sO_X  
\rightarrow  B^1 \rightarrow 0\eqno{(1)}$$, 
$$0 \rightarrow B^1 \rightarrow Z^1  \rightarrow \Omega^1
\rightarrow 0\eqno{(2)}$$,  
$$0 \rightarrow Z^1 \rightarrow  F_*\Omega^1 \rightarrow B^2  \rightarrow 
0\eqno{(3)}$$ 
This are sequences obtained from the Cartier operator applied to the 
DeRham complex $F_.(\Omega^._X)$. Tensor all the $3$ sequences by $\sEnd 
V(-n)$, and take $H^0$. From (2), 
we  get  $H^1(B^1 \otimes \sEnd V( - n))$ injects into  $H^1(Z^1 \otimes
\sEnd V( - n))$ for a uniform $n$. But by $3$, we see that $H^1(Z^1  
\otimes \sEnd V(-n))$ injects into  $H^1(F_*\Omega^1 \otimes \sEnd V(-n))$
for a uniform $n$. But the last cohomology group vanishes for $n >> 0$,
with $n$ uniform (Lemma 2.2  and [S,Th.3.5]). Therefore ,$H^1(B^1 \otimes
\sEnd V(-n))$ vanishes for $n >> 0$, with $n$ \emph {uniform}. Now look at 
$$0 \rightarrow \sO_X  \rightarrow  F_*\sO_X  \rightarrow B^1 \rightarrow 
0\eqno{(4)}$$ Tensor with $\sEnd V(-n)$ and take cohomology.As $H^1(B^1 
\otimes \sEnd V(-n))$ vanishes for $n >> 0$, one gets that $H^2(\sEnd V 
(-n))$ injects into $H^2(\sEnd V(-n))^{p^r}$ for all $r >> 0$, and for $n$ 
uniform. But this last group \emph {vanishes} for $n >> 0$, and uniform,
and for all $r$(Lemma 2.2 again). (This is where the hypothesis dim$X 
\geq 3$ is used).Now finally look at 
$$0 \rightarrow \sO_X(-3n) \rightarrow \sO_X(-2n)  \rightarrow \sO_Y(-2n)
\rightarrow 0\eqno{(5)}$$ and tensor with $\sEnd V$. We get 
$H^1(\sO_Y(-2n)
\otimes \sEnd V))$  vanishes for $n >> 0$, and n uniform. Hence,  
 we get $H^0(\sEnd V \otimes \Omega^1_Y(-n)) = 0$, with n uniform ,
 which completes the proof of Claim 1.
 
Claim (3): There is a uniform n with the property : let $Y \in |nH|$ and
 assume that $W \in FT(Y)$ has been lifted to $V \in FT(X)$. Then if 
$W$ has a  connection on $Y$ ,then $V$ on $X$ has a connection.

Proof of Claim (3):Look at $$0 \rightarrow \sO_Y(-n) \rightarrow 
\Omega^1_X/Y \rightarrow  \Omega^1_Y \rightarrow  0\eqno{(1)}$$ and
$$0 \rightarrow \Omega^1_X(-n) \rightarrow \Omega^1_X  \rightarrow
\Omega^1_X/Y \rightarrow  0\eqno{(2)}$$ 
Tensor with $\sEnd V$.  From (2) we get $H^1(\sEnd V  \otimes 
\Omega^1_X)$ injects into $H^1(\sEnd V \otimes \Omega^1_X/Y)$.
And from (1) and the proof of claim (2), we get $H^1(\sEnd V \otimes 
\Omega^1_X/Y)$ injects into $H^1(\sEnd V \otimes \Omega^1_Y)$. But $V$ 
restricted to $Y$ is $W$. Hence if $W$ has a connection ,so does $V$. 

Claim (4) : If this connection on $W$ is integrable  ,then the 
connection  on $V$ is also integrable.

Proof of Claim (4) : Look at $$0 \rightarrow  \Omega^1_Y(-n) \rightarrow
\Omega^2_X/Y \rightarrow  \Omega^2_Y \rightarrow 0\eqno{(1)}$$
and $$0 \rightarrow \Omega^2_X(-n) \rightarrow  \Omega^2_X \rightarrow
\Omega^2_X/Y \rightarrow 0\eqno{(2)}$$
Tensor with $\sEnd V$ and take $H^0$. One knows that $H^0(\sEnd V 
\otimes \Omega^1_Y( -n)$ vanishes for $n >> 0$ with $n$ uniform ,by the 
proof of Claim 1. Similiarly, $H^0(\sEnd V \otimes \Omega^2_X(- n)$ 
also vanishes for $n >> 0$, and $n$ uniform.So  $H^0(\sEnd V \otimes 
\Omega^2_X) \rightarrow  H^0(\sEnd V \otimes \Omega^2_Y)$
is injective for $n >> 0$, n \emph {uniform}. So if $W$ has an integrable 
connection, so does $V$. Finally, the $p$-curvature of $V$ is an element
of $H^0(\sEnd V \otimes  F^*\Omega^2_X)$ which injects into $H^0( \sEnd V
\otimes F^*\Omega^2_Y)$. So if the $p$-curvature on $W$ is $0$, so is the
$p$-curvature on $V$.

  Claims 1-4 imply that on $X$,if $V$ restricts to $W$ on $Y$, with degree 
$Y = n$ with n uniform and $W$ has a $p$-flat connection, then $V$ also 
has a $p$-flat connection. So if $W$ on $Y$ descends to $W_1$ ,
then $V$ also descends under $F$ to $V_1$, and that $V_1$ restricts to 
$W_1$ on $Y$. This continues to hold for $V \in C_t(X)$,
restricting to $W \in C_t(Y)$, any $t$, for $Y$  degree a uniform n, 
depending \emph{only} on $X$.   
Hence the canonical map of Tannaka Categories : $C_t(X) \rightarrow 
C_t(Y)$ , given by restriction from $X$ to $Y$, induces an 
\emph{isomorphism}:$$G_t(Y) \rightarrow G_t(X),$$ for all $t$.But 
 $$\pi^{loc}(X) \simeq \lim_{\leftarrow t}G_t(X)$$ and similiarly for 
$Y$. Since there are only finitely many choices of $n$, depending 
\emph {only on $X$}, this completes the proof of Lemma $3.5$ and hence of 
Theorem $3.1$.     
\end{proof}

\begin{Remark} It is also interesting to determine if the category
of F-trivial vector bundles ,on a smooth $X$,is \emph {$m_0$- regular}. 
This is the case if dim $X \leq 3$.
\end{Remark}

\begin{Remark} The discerning reader will notice that the only property of
F-trivial bundles which is used is the following : the set of 
\emph{isomorphism classes} of stable bundles, which occur in a stable 
filtration of $F^{n*}(V),n = 0,1,....$ is only \emph{finite} in number.
  
 If $V$ is only assumed to be \emph{essentially finite}[N1,p.82],then 
$\exists$  a Galois etale covering $ \pi : Z \to X$ such that $F^{m*}V$ is 
trivial on $Z$, for some $m$. Now one sees using [D1, Thm 2.3.2.4],that 
the set of \emph{stable} components of $F^{n*}(V), n = 0,1...$ is again  
\emph{finite}.  So  all the  proofs and propositions carry over with 
$\pi^{loc}X$  replaced  by $\pi(X)$,making use of Lemma 3.3.  We leave 
the details as an exercise. 

\end{Remark}

\end{document}